\documentclass[12pt]{article}

\usepackage{graphicx}
\usepackage{amsmath,amsthm,amssymb,color}
\usepackage[noadjust,sort]{cite}
\usepackage[mathscr]{euscript}
 \usepackage[normalem]{ulem}

\normalsize

\theoremstyle{plain}
\newtheorem{Theorem}{Theorem}[section] %
\newtheorem{Lemma}[Theorem]{Lemma}

\newtheorem{Corollary}[Theorem]{Corollary}

\theoremstyle{definition}
\newtheorem{Remark}[Theorem]{Remark}

\def\calF{\mathcal{F}}

\def\calL{\mathcal{L}}
\def\calH{\mathcal{H}}

\def\leq{\leqslant}
\def\geq{\geqslant}

\theoremstyle{definition}

\newtheorem{Problem}{Problem}[section]

{\par\noindent{\it Proof of}} 
{\hfill$\vspace{5mm}\scriptstyle\blacksquare$} 

\numberwithin{equation}{section} 
\numberwithin{figure}{section} 
\numberwithin{table}{section} 

\def\supp{\operatorname{supp}}

\def\R{\mathcal{R}}

\def\Reals{{\mathbb{R}}}

\def\Naturals{{\mathbb{N}}}

\def\st{\,:\,}

\def\dfrac#1#2{\lower0.15ex\hbox{\large$\textstyle\frac{#1}{#2}$}}

\begin{document}

\setcounter{page}{1}

\markboth{M. Isaev, R.G. Novikov}{
Reconstruction
from the Fourier transform on the ball via PSWFs}

\title{
Reconstruction
from the Fourier transform  on the ball via prolate spheroidal wave
functions
 \thanks{The first author's research is  supported by    the Australian Research  Council  Discovery Early Career Researcher Award DE200101045. 
}
}
\date{}
\author{ 
Mikhail Isaev\\
\small School of Mathematics\\[-0.8ex]
\small Monash University\\[-0.8ex]
\small Clayton, VIC, Australia\\
\small\texttt{mikhail.isaev@monash.edu}
\and
Roman G. Novikov\\
\small CMAP, CNRS, Ecole Polytechnique\\[-0.8ex]
\small Institut Polytechnique de Paris\\[-0.8ex]
\small Palaiseau, France\\
\small IEPT RAS, Moscow, Russia\\
\small\texttt{novikov@cmap.polytechnique.fr}
}

\maketitle

\begin{abstract}
	We give new formulas for finding a   compactly supported function $v$ on $\Reals^d$, $d\geq  1$,  
	  from its   Fourier transform  $\calF v$ given within the ball $B_r$. For the one-dimensional case, these formulas are
	 based on the theory of prolate spheroidal wave
functions (PSWF's).  In multidimensions,   well-known results of the Radon transform theory  reduce the problem to the one-dimensional case.  
Related results on stability and  convergence rates are also given. 

\end{abstract}

\section{Introduction}\label{S:intro}

Following D. Slepian, H. Landau,  and H. Pollak (see, for example, the survey paper \cite{Slepian1983}),  we consider the compact integral operator $\calF_c$ on $\calL^2([-1,1])$ defined by
 \begin{equation}\label{def:Fc}
 	\calF_c[f] (x) := \int_{-1}^1 e^{i c xy} f(y)dy,
 \end{equation}
 where $f$ is a test function and the parameter $c>0$ is the bandwidth.   
 Let $\Naturals:=\{0,1\ldots\}$.
The eigenfunctions 
 $(\psi_{j,c})_{j \in \Naturals}$    of $\calF_c$ 
 are   \emph{prolate spheroidal wave
functions} (PSWFs).  These functions are real-valued and  form an orthonormal basis in $\calL^2([-1,1])$. Let $(\mu_{j, c})_{j \in \Naturals}$ denote the corresponding eigenvalues. It is known that  all  these eigenvalues are  simple and non-zero,  so we can assume that  
$0<|\mu_{j+1,c}| < |\mu_{j,c}|$ for all $j \in \Naturals$.

 The properties of  $(\psi_{j,c})_{j \in \Naturals}$ and $(\mu_{j, c})_{j \in \Naturals}$ are recalled in Section \ref{S:prolate} of this paper.  In particular, we have that 
\begin{equation}\label{Fc-dec} 
	\calF_c [f] (x) = \sum_{j \in \Naturals} \mu_{j,c}\psi_{j,c}(x) \int_{-1}^1 \psi_{j,c} (y) f(y) dy,
	\end{equation}
	and, for $g = \calF_c [f]$,
	\begin{equation}\label{f:inverse}
	  \calF_{c}^{-1} [g](y) =    \sum_{j \in \Naturals} \dfrac{1}{\mu_{j,c}}\psi_{j,c}(y) \int_{-1}^1 \psi_{j,c} (x) g(x)dx,
\end{equation}
where  
$\calF_{c}^{-1}$ is the inverse operator, that is  $\calF_{c}^{-1} [\calF_c [f]] \equiv f $ for all $f \in  \calL^2[-1,1]$.

   The operator $\calF_c$  appears naturally in the theory of the classical Fourier transform
    $\calF$  defined in the multidimensional case $d\geq 1$ by
 \begin{equation}
 \label{eq:Fourier}
 	\calF[v] (p) := 
 	\dfrac{1}{(2\pi)^d}\int\limits_{\mathbb{R}^d} e^{i pq } v(q) dq, \qquad  p\in \mathbb{R}^d,
  \end{equation}
  where   $v$ is a complex-valued test function on $\Reals^d$.  To avoid any possible confusion with $\calF_c$, we employ the simplified notation $\hat{v} := \calF[v]$  throughout the paper. 
  Let
 \begin{equation}
  B_{\rho} := \left\{q\in \mathbb{R}^d :  |q| < \rho \right\}, \qquad
 	\text{for any  $\rho>0$.} \nonumber
 \end{equation}
  We consider  the following inverse problem. 
 \begin{Problem}\label{Problem}
Let   $d\geq 1$ and  $ r,\sigma>0$.
 Find    $v \in \calL^2(\Reals^d)$
 from   $\hat{v}$  given on the ball  $B_r$ (possibly with some noise), under a priori assumption that $v$ is   supported in $B_{\sigma}$.
 \end{Problem}

 Problem \ref{Problem} is a classical problem of the Fourier analysis, inverse scattering, and   image processing;
see, for example,   \cite{LRC1987, AMS2009, BM2009, Papoulis1975, CF2014, Gerchberg1974, IN2020, IN2020+}  and references therein. 
In the present work, we suggest a new approach to Problem \ref{Problem},  proceeding from the singular value decomposition formulas \eqref{def:Fc}, \eqref{Fc-dec} and further results
of the PSWF theory. 
Surprisingly, to our knowledge,  the  PSWF theory was omitted in the context of  Problem \ref{Problem} in the literature even though it is quite natural. 
In particular, in dimension $d=1$,  Problem \ref{Problem}  reduces to finding a function $f\in \calL^2([-1,1])$ from  $\calF_c[f]$ (possibly with some noise). 

%

%
%
%
%
%
%
%
%
%
%
%
%
%
%
%
%
%
  
In multidimensions, in addition to the PSWF theory,  we use   inversion methods  for  the classical Radon transform $\mathcal{R}$ ; see, for example   \cite{Naterrer2001, Radon}.
Recall that  $\mathcal{R}$ is  defined by 
 \begin{equation}\label{def:R}
 	\mathcal{R} [v]  (y, \theta) := \int_{q\in \Reals^d \st q \theta =y } v(q) dq, \qquad y\in \Reals,\ \theta \in \mathbb{S}^{d-1},
 \end{equation}
 where   $v$ is a complex-valued  test function on $\Reals^d$, $d \geq 1$.   In the present work,  for simplicity,  we define the inverse Radon transform $ \R^{-1}$  via the projection theorem; see formula \eqref{inv:radon} for details.  
%



 \begin{Theorem}\label{T1}Let $d\geq 1$,  $r,\sigma>0$ and $c = r\sigma$.
  Let $v\in \calL^2(\Reals^d)$ and $\supp v \subset  B_{\sigma}$. 
  Then,   its Fourier transform $\hat{v}$ restricted to $B_r$ determines $v$ via the following formulas:
  \begin{align*}
  v(q) &=     \R^{-1} [f_{r, \sigma}]( \sigma^{-1}q), \qquad  q\in \Reals^d,\\
  f_{r,\sigma}(y,\theta) &:=
 	 \begin{cases}
 \calF_c^{-1}[g_{r,\theta}](y), &\text{if } y\in [-1,1]\\
 0, & \text{otherwise},  
 \end{cases} 
 \\
 g_{r,\theta} (x)&:=      \left(\dfrac{2\pi}{\sigma}\right)^d \hat{v} (r x\theta), \qquad  x\in [-1,1],\ \theta \in \mathbb{S}^{d-1},
  \end{align*}
 where     $\calF_c^{-1}$ is     defined by \eqref{f:inverse}
 and   $ \R^{-1}$ is the inverse Radon transform.
 \end{Theorem}
\begin{Remark}\label{rem1}
	For $d=1$, the formulas of Theorem \ref{T1} reduce to 
	\begin{align*}
		v(q)  =    \calF_c^{-1}[g_{r}]( \sigma^{-1}q), 
		\qquad 
		g_r(x) := \dfrac{2\pi}{\sigma}\hat{v}(rx), 
	\end{align*}
	where $q\in (-\sigma,\sigma)$ and $x\in [-1,1]$.
\end{Remark} 
 We  prove Theorem \ref{T1} in Section \ref{S:T1}.

  Unfortunately, the reconstruction procedure given in Theorem \ref{T1} 
  and Remark \ref{rem1}   is severely unstable. The reason is that   the numbers  $(\mu_{j, c})_{j \in \Naturals}$  decay superexponentially as $j \rightarrow \infty$; see formulas \eqref{eq:eigenrel} and \eqref{eigenestimate}. To overcome this difficulty, we approximate 
 $ \calF_c^{-1}$ by the operator  $\calF_{n,c}^{-1}$ defined by
 \begin{equation}\label{def:Fnc}   
 	\calF_{n,c}^{-1} [w] (y) :=  \sum_{j=0}^n \dfrac{1}{\mu_{j,c}}\psi_{j,c}(y) \int_{-1}^1 \psi_{j,c} (x) w(x)dx. 
 \end{equation}
 Note that  \eqref{def:Fnc} correctly defines the operator   $\calF_{n,c}^{-1}$   on $\calL^2([-1,1])$ for any $n \in \Naturals$.   
 Let 
 \begin{equation}\label{def:pi_n}
 	\pi_{n,c}[f]:=  \sum_{j =0 }^n  \hat{f}_{j,c}  \psi_{j,c},  
\qquad
 	\hat{f}_{j,c} := \int_{-1}^1 \psi_{j,c} (y) f(y) dy.
 \end{equation}
 That is,  $\pi_{n,c}[\cdot]$ is the orthogonal projection  in $\calL^2([-1,1])$  onto the span of   the first $n+1$ functions
  $(\psi_{j,c})_{j  \leq  n}$.

  \begin{Lemma}\label{L:general}
    Let $f,w \in \calL^2([-1,1])$  and 
   $
    	\| \calF_c [f] - w\|_{\calL^2} \leq \delta  
   $ for some $\delta\geq 0$. Then, for any $n \in \Naturals$,
%
%
 \begin{equation}\label{eq:general}
 	\|f - \calF_{n,c}^{-1} [w]\|_{\calL^2([-1,1])}
 	\leq  \dfrac{\delta }{|\mu_{n,c}|}   + \|f - \pi_{n,c} [f]\|_{\calL^2([-1,1])}.
 \end{equation}
 \end{Lemma}

Estimates of the type \eqref{eq:general} are of general nature for operators admitting a singular value decomposition like \eqref{Fc-dec}. For completeness of  the presentation, 
we prove Lemma~\ref{L:general} in Section~\ref{S:Lemma}. 
 Combining Theorem \ref{T1},  Remark \ref{rem1},  Lemma \ref{L:general},   inversion methods for the Radon transform $\mathcal{R}$, and known estimates of the PSWF theory for $|\mu_{n,c}|$ 
 and  $\|f - \pi_{n,c} [f]\|_{\calL^2}$ (see Section \ref{S:prolate})
  yields numerical methods for  Problem \ref{Problem}.  
  In this connection, in the present work we give a regularised version of the reconstruction procedure of Theorem \ref{T1}; see Theorem  \ref{T:multi}, Theorem \ref{T:detailed}
 and Corollary \ref{C:main}.

  For  $\alpha,\delta \in(0,1)$, let  
  \begin{equation}\label{def:n-star}
   n^*  = n^*(c, \alpha,\delta) =   \left\lfloor 3+ \tau \dfrac{ec}{4}\right\rfloor, 
 \end{equation}
   where $\lfloor\cdot \rfloor$ denotes the floor function and
   $\tau = \tau(c,\alpha,\delta) \geq 1$ is the solution of  the equation
   \begin{equation}\label{eq-tau}
   \tau \log \tau =    \dfrac{4}{ec}  \alpha  \log  (\delta^{-1}).
  \end{equation}

 Let 
 \begin{equation}\label{def:L2r}
 \begin{aligned}
 \calL^2_r
 	&:= \{w \in \calL^2 (B_r) \st  \|w\|_r <  \infty\},\\
 	\|w\|_{r}&:=  \left( \int_{B_r} p^{1-d}|w(p)|^2  dp  \right)^{1/2}.
 \end{aligned}
  \end{equation}
  
 \begin{Theorem}\label{T:multi}
 	Let the assumptions of Theorem \ref{T1} hold and $v \in \calH^\nu(\Reals^d)$ for some  $\nu \geq 0$  (and $\nu >0$ for $d=1$). Suppose that $w \in \calL^2_r $ and  $\|w  -  \hat{v}\|_{r} \leq \delta N$ for some $\delta \in (0,1)$. Let  $\alpha \in (0,1)$ and $n^*$ be defined  by \eqref{def:n-star}. Let  
 		\begin{align*}
 			 v^{\delta} (q) &:=  \mathcal{R}^{-1}\left[u_{r,\sigma}\right] (\sigma^{-1}q), \qquad  q \in \Reals^d,\\
      		u_{r,\sigma}(y,\theta)&:= 	 \begin{cases}
 \calF^{-1}_{n^*,c}[w_{r,\theta} ](y), &\text{if } y\in [-1,1],\\
 0, & \text{otherwise},  
 \end{cases} \\
      		w_{r,\theta}(x)&:= \left(\dfrac{2\pi}{\sigma}\right)^d  w(r x\theta), \qquad  x\in [-1,1],\ \theta \in \mathbb{S}^{d-1}.
 	\end{align*} 
 	Then, for any    $\beta \in (0, 1-\alpha)$ and any $\mu \in (0,\nu+\frac{d-1}{2})$, 
      \begin{equation}\label{eq:multi}
      		 \|v - v^{\delta}\|_{\calH^{-(d-1)/2}(\Reals^d)}
      		  \leq \kappa_1  N  \delta^{\beta} +  \kappa_2   \|v\|_{\calH^\nu(\Reals^d)}  \left( \log \delta^{-1}\right)^{-\mu}, 
      \end{equation}
      where  $\kappa_1 = \kappa_1(c, d, r, \sigma,\alpha,\beta)>0$ and $\kappa_2 = \kappa_2(c, d, r,\sigma,\alpha,\nu,\mu)>0$.
 \end{Theorem}

 Similarly to Remark \ref{rem1}, the statement of  Theorem \ref{T:multi} simplifies significantly for the case $d=1$; see Corollary \ref{C:main}. 
 We prove  Theorem \ref{T:multi} in  Section \ref{S:multi}.

 The parameter $N$ from Theorem \ref{T:multi} can be considered as an a priori upper bound for $\|\hat{v}\|_r$.  Indeed, the assumption $\|w  -  \hat{v}\|_{r} \leq \delta \|\hat{v}\|_r$ is natural. If the noise level is such that $\|w  -  \hat{v}\|_{r} \geq   \|\hat{v}\|_r$,  then the given data $w$ tells about $v$ as little as the trivial function $w_0 \equiv 0$. An accurate reconstruction is hardly possible in this case, since it is  equivalent to no   data given at all.

 The function $v^\delta$  in Theorem \ref{T:multi}  is not  compactly supported, in general; see also the related remark about  $v$ after Lemma \ref{L:H-H}.
 Nevertheless, only $v^\delta$  restricted to $B_\sigma$  is of interest under the assumptions of Theorem \ref{T:multi}.

 Our stability estimate \eqref{eq:multi} is  given in $\calH^s$  with  $s\leq 0$. 
 One can improve the regularity in such estimates using  the apodized reconstruction $\phi * v^\delta$, where  $*$ denotes the convolution operator
and $\phi$ is an appropriate sufficiently regular non-negative compactly supported function with 
$\|\phi\|_{\calL^1(\Reals^d)}=1$;  see, for example,  \cite[Section~6.1]{IN2020}.
In particular, \eqref{eq:multi} implies estimates for  
$\phi* v- \phi * v^\delta$  in $\calH^t$  with $t\geq 0$.

Applying Theorem \ref{T:multi} with $v:=v_1- v_2$ and $w\equiv 0$, we get the following result.

\begin{Corollary}\label{C:multi2} Let the assumptions of Theorem 1.1 hold for 
$ v:=v_1- v_2$. Let  $v_1- v_2  \in \calH^\nu(\Reals^d)$
for some  $\nu \geq 0$  (and $\nu >0$ for $d=1$).
Suppose that  $\|\hat{v}_1\ -  \hat{v}_2\|_{r} \leq \delta N$ for some $\delta \in (0,1)$ and $N>0$. Let  $\alpha \in (0,1)$.
Then, for any    $\beta \in (0, 1-\alpha)$ and any $\mu \in (0,\nu+\frac{d-1}{2})$,
      \begin{equation}\label{eq:multi2}
               \|v_1 - v_2 \|_{\calH^{-(d-1)/2}(\Reals^d)}
                \leq \kappa_1  N  \delta^{\beta} +  \kappa_2   \| v_1 - v_2 \|_{\calH^\nu(\Reals^d)}  \left( \log \delta^{-1}\right)^{-\mu},
      \end{equation}
where   $\kappa_1 = \kappa_1(c,d,r,\sigma,\alpha,\beta)$  and   $\kappa_2=
\kappa_2(c,d,r,\sigma,\alpha,\nu,\mu)$   are  the same as in   \eqref{eq:multi}.
\end{Corollary}

 The present work continues studies of 
 \cite{IN2020,IN2020+},  where we approached Problem~\ref{Problem} via  a H\"older-stable extrapolation
of $\hat{v}$ from $B_r$  to a larger ball, using  truncated series of Chebyshev polynomials.
The reconstruction of the present work is essentially different;  in particular, it does not use any extrapolation. 
However, the resulting stability estimates are analogous for both reconstructions.
In particular,  estimate  \eqref{eq:multi} resembles   \cite[Theorem~3.1]{IN2020} in dimension $d=1$ and resembles  \cite[Theorem~3.2]{IN2020+}
(with $s= -\frac{d-1}{2}$ and $\kappa=1$) in dimension $d\geq  1$;
  estimate  \eqref{eq:multi2} resembles   \cite[Corollary 3.3]{IN2020} in dimension $d=1$ and resembles  \cite[Corollary~3.4]{IN2020+}
(with $s= -\frac{d-1}{2}$ and $\kappa=1$) in dimension $d\geq  1$.
Note also that,  in the domain of  coefficient inverse problems, estimates of the form \eqref{eq:multi} and \eqref{eq:multi2}  are known as H\"older-logarithmic stability estimates;
see  \cite{IN2020,IN2020+, IN2013++, HH2001, HW2017}  and references therein.

The main advantages of the present work in comparison with  \cite{IN2020,IN2020+} are the following:

\begin{itemize}
\item We allow the "noise"  in Problem \ref{Problem} to be from a larger space
$\calL^2_r$  defined by \eqref{def:L2r} in contrast with $\calL^\infty$.

\item We  use   the straightforward 
  formulas  \eqref{f:inverse}, \eqref{def:Fnc},   \eqref{eq:general} 
  in place of  the  roundabout way   that requires   extrapolation of  
$\hat{v}$ from $B_r$  to a larger ball  and  leads to additional numerical issues.

\end{itemize}

On the other hand,  the advantages  of  \cite{IN2020,IN2020+} in comparison with the present work include: 
  explicit expressions for quantities like $\kappa_1$ and $\kappa_2$ in  \eqref{eq:multi};
 more advanced norms $\|\cdot\|$ for reconstruction errors like $v - v^{\delta}$ in  \eqref{eq:multi}, where $\|\cdot\| = 
 \|\cdot\|_{\calL^2(\Reals^d)}$
 in   \cite{IN2020} and  $\|\cdot\| = 
 \|\cdot\|_{\calH^s(\Reals^d)}$ with  any $s \in (-\infty, \nu)$
  in  \cite{IN2020+}.  The reason  is 
  purely due to the fact that the PSWFs theory is still
less developed than the theory of Chebyshev polynomials and the classical Fourier transform theory.  
In connection with further developments  in the PSWFs theory
that would improve the  results of the present work on Problem \ref{Problem}, see Remarks  \ref{R21}, \ref{R22}, and \ref{R23} in Section~\ref{S:prolate}.


 Note also that  the functions  $(\psi_{j,c})_{j \in \Naturals}$  for large $j$, 
yield a new example of exponential instability for Problem \ref{Problem} in dimension $d=1$. 
This instability behaviour follows from  the properties of  $\psi_{j,c}$ and $\mu_{j,c}$ recalled in  Section~\ref{S:prolate} and  the  result formulated in Remark~\ref{R22}.
%
%
However, known estimates for the derivatives of PSWFs  do not allow yet to say that this example is more strong than
the example constructed in \cite[Theorem 5.2]{IN2020}.

The aforementioned possible developments in the PSWFs theory and
further development of the approach of the present work to Problem \ref{Problem},
including its numerical implementation, will be addressed in further articles.

The further structure of the  paper is as follows.
 Some prilimary results are recalled in Section \ref{S:preliminaries}. In Section \ref{S:1D}, we  prove our estimates  in dimension $d=1$ modulo a technical lemma, 
 namely,  Lemma \ref{L:delta-mu}. 
 In Section \ref{S:multi}, we prove Theorem~\ref{T1}, Theorem~\ref{T:multi} and  Corollary~\ref{C:multi2}
 based on the  results given in Sections \ref{S:preliminaries} and  \ref{S:1D}.
  In Section  \ref{S:detailed}, we prove Lemma \ref{L:delta-mu}. 

\section{Preliminaries}\label{S:preliminaries}
 In this section, we recall some  known results on PSWFs  and on the Radon transform that we will use in the proofs of Theorems \ref{T1} and \ref{T:multi}. In addition, we   prove Lemma \ref{L:general}
 and give a stability estimate for the inverse Radon transform; see Lemma \ref{L:H-H}.

\subsection{Prolate spheroidal wave functions}\label{S:prolate}
In connection with the facts presented in this subsection we refer to  
 \cite{Slepian1983, BK2017, BK2017+, RX2007, Wang2010, STR2006}   and references therein.


Originally, the prolate spheroidal wave functions  $(\psi_{n,c})_{n \in \Naturals}$ were discovered as the eigenfunctions of the  
 following spectral problem:
\begin{equation}\label{L-prob}
	\mathcal{L}_c \psi = \chi \psi, \qquad \psi \in C^2([-1,1]),   
\end{equation}
where $\chi$ is the spectral parameter and 
\begin{align*}	
	\calL_c [\psi] := - \dfrac{d}{dx} \left[(1-x^2) \dfrac{d \psi}{ dx}\right] + c^2 x^2 \psi.
\end{align*} 
We also consider  the operator $\mathcal{Q}_c$  defined  on $\calL_2([-1,1])$ by 
\begin{equation}
	\mathcal{Q}_c[f](x):=\frac{c}{2\pi} \calF_c^* \left[\calF_c [f]\right](x)= \int_{-1}^1 \dfrac{\sin c(x-y)}{ \pi (x-y)} f(y) dy, 
\end{equation}
where $\calF_c^*$ is the conjugate operator to $\calF_c$  defined by \eqref{def:Fc}. 
  The prolate spheroidal wave functions  $(\psi_{n,c})_{n \in \Naturals}$  are  eigenfunctions for 
  problem \eqref{L-prob} and for  both operators 
 $\calF_c$ and $\mathcal{Q}_c$.

Let $(\chi_{n,c})_{n \in \Naturals}$ denote the eigenvalues of problem \eqref{L-prob}. It is  known that $(\chi_{n,c})_{n \in \Naturals}$ are real, positive, simple, that is, one can assume that 
\[
	0<\chi_{n,c}<\chi_{n+1,c}, \qquad \text{for all } n \in \Naturals. 
\]  
In addition,   the following  estimates hold:
\begin{equation}\label{chi<chi}
		n(n+1) < \chi_{n,c}<n(n+1)+c^2.
\end{equation}
 If  $ \mu_{n,c} $ and $ \lambda_{n,c}$ are the corresponding eigenvalues of   $\calF_c$  and $\mathcal{Q}_c$, respectively, then
%
%
%
%
\begin{equation}\label{eq:eigenrel}
\begin{aligned}
	\mu_{n,c} = i^n \sqrt{\dfrac{2\pi}{c} \lambda_{n,c}} \quad \text{ and } \quad
		1>\lambda_{n,c} > \lambda_{n+1,c}>0.
\end{aligned}
\end{equation}
Furthermore,  each $\lambda_{n,c}$ is non-decreasing with respect to $c$. Using also \cite[formula  (6)]{BK2017},  we  find that
\begin{equation}\label{eq:biglam}
  \left\lfloor\frac{2c}{\pi}\right\rfloor-1\leq \Big|\{n \in \Naturals \st \lambda_{n,c}\geq 1/2\}\Big|\leq  \left\lceil \frac{2c}{\pi}\right\rceil+1.
 		\end{equation}	
%
%
where $\lfloor \cdot \rfloor$ and  $\lceil \cdot \rceil$  denote the floor and the ceiling functions, respectively,  and $|\cdot|$ is the number of elements.
  We also employ the following estimate from   
\cite[Corollary 3]{BK2017}:   
for $n \geq \max\{3, \frac{2c}{\pi}\}$,
	\begin{equation}\label{eigenestimate}
				A(n,c)^{-1}   
			e^{-2\tilde n (\log  \tilde{n}-\kappa)} \leq  \lambda_{n,c} \leq 
		A(n,c) 	e^{-2\tilde n (\log  \tilde{n}-\kappa)}, 
	\end{equation}
	 where  $\nu_1\geq 1$, $\nu_2,\nu_3 \geq 0$ are some fixed  constants, 
	 	\[
		A(n,c):= \nu_1 n^{\nu_2} \left(\frac{c}{c+1}\right)^{-\nu_3} e^{(\pi c)^2/4n}.
		\]
	and  
 \begin{equation}\label{def:kappa}
	  \kappa := \log \left( \dfrac{ec}{4}\right), \qquad \tilde{n} = \tilde{n}(n):= n+\dfrac 12.
\end{equation}
\begin{Remark}\label{R21}
Apparently, proceeding from the approach of \cite{BK2017}, one can give explicit values for the constants $\nu_1$, $\nu_2$,  $\nu_3$ in the expression for  $A(n,c).$
\end{Remark}
\noindent
We also recall from \cite[formula (11)]{STR2006}  that, for all $n \in \Naturals$ and $c>0$, 
\begin{equation}\label{norm:Linfty}
	\max_{0 \leq j \leq n} \max_{|x|\leq 1} 
	|\psi_{j,c}(x)|\leq 2\sqrt{n}. 
\end{equation}
\begin{Remark}\label{R22}
 Proceeding from  \eqref{L-prob},  \eqref{chi<chi}, and \eqref{norm:Linfty}, one can show  that,   for  any $m\in \Naturals$,
 \[ 
 \|\psi_{n,c}\|_{C^m[-1,1]} = O ( n^{2m+1/2}) \qquad \text{as  $n \rightarrow \infty$.}
 \]
\end{Remark}

	 Next, we recall results on the spectral approximation by PSWFs in Sobolev-type spaces; see  \cite{Wang2010}. For a real $\nu\geq 0$,  let
	 \begin{equation}
	 	\widetilde\calH^\nu_c([-1,1]):= \left\{ f \in \calL^2([-1,1]) \st 
	 	 \|f\|_{\widetilde\calH^\nu_c} <\infty \right\},
	 \end{equation}
	   where
	   \begin{align*}
	   	 \|f\|_{\widetilde\calH^\nu_c([-1,1])} := \left(\sum_{n\in \Naturals} (\chi_{n,c})^{\nu} |\hat{f}_{n,c}|^2\right)^{1/2}
		\qquad   \hat{f}_{n,c}:= \int_{-1}^1 \psi_{n,c} (y) f(y)dy.
	   \end{align*}
 Recall  from \eqref{def:pi_n} that 
 \[
 	\pi_{n}[f]=  \sum_{j =0 }^n  \hat{f}_{j,c}\psi_{j,c}(x), \qquad n\in \Naturals.
 \]
 Note that 
  $\pi_{n}[f] \rightarrow f$ as $n \rightarrow \infty$ since    	 $(\psi_{j,c}(x))_{j \in \Naturals}$  form an orthonormal basis in $\calL^2([-1,1])$.
  Furthermore,   for any $0 \leq \mu \leq \nu$,
  \begin{equation}\label{eq:pi1}	
  	\left\|f-\pi_{n}[f] \right\|_{\widetilde\calH^\mu_c([-1,1])} \leq n^{\mu-\nu}   \|f\|_{\widetilde\calH^\nu_c([-1,1])}.
  \end{equation}
  The standard Sobolev space  $\calH^\nu[(-1,1)]$ is  embedded in  $\widetilde\calH^\nu_c([-1,1])$. In fact, we have that
  \begin{equation}\label{eq:pi2}
  		\|f\|_{\widetilde\calH^\nu_c[(-1,1)]}   \leq  C(1+c^2)^{\nu/2} \|f\|_{\calH^\nu([-1,1])},
  \end{equation}
  where $C$ is a constant independent of $c$ and $f$ assuming that $c \geq c_0>0$.
\begin{Remark}\label{R23}
Proceeding from the results of \cite{Wang2010}, one  can obtain an explicit  estimate for
the constant $C=C(c_0, \nu)$ in  \eqref{eq:pi2}.
Besides, one can establish an  upper bound for 
$\|\varphi f\|_{\calH^\nu ([-1,1])}$  in terms of 
$\| f \|_{\widetilde{\calH}^\nu_c([-1,1])}$, for fixed $\nu>0$,  where 
$\varphi$ is a  smooth real-valued function
appropriately vanishing at the ends of the interval $[-1,1]$ and non-vanishing elsewhere.
\end{Remark}

\subsection{Proof of  Lemma \ref{L:general}}\label{S:Lemma}

First, we observe that 
\[\calF_{n,c}^{-1}[\calF_c[f] ]  =  \pi_n [f].\]
 Using also the linearity of $\calF_{n,c}^{-1}$, we derive 
\[
	f - \calF_{n,c}^{-1}[w]    
	=  f  -\pi_n [f] + \calF_{n,c}^{-1}[\calF_c[f] ] -  \calF_{n,c}^{-1}[w]
	= 
	 f - \pi_n [f] +   \calF_{n,c}^{-1} [u],
\] 
where $u := \calF_c[f] - w$. 
Therefore,
\[
	\left\| f - \calF_{n,c}^{-1}[w]\right\|_{\calL^2([-1,1])}  \leq 
	 \left\|\calF_{n,c}^{-1} [u]\right\|_{\calL^2([-1,1])}+ \left\|f - \pi_n [f] \right\|_{\calL^2([-1,1])} .
\]
 Due to \eqref{eq:eigenrel}, we have that
    $|\mu_{j,c}| \geq  |\mu_{n,c}|$ for all $j \leq n$. Using  also the orthonormality of the basis 
    $(\psi_{j,c})_{j \in \Naturals}$ in $\calL^2([-1,1])$, we estimate
%
      	\begin{align*}
   	\| \calF_{n,c}^{-1} [u] \|_{\calL^2([-1,1])}^2 &= \left\|  \sum_{j =0}^n \dfrac{1}{\mu_{j,c}}\psi_{j,c}(\cdot) \int_{-1}^1 \psi_{j,c} (x) u(x)dx   \right\|_{\calL^2([-1,1]) }^2   
   		 \\  &  =
   		   \sum_{j =0}^n   \dfrac{1}{|\mu_{j,c}|^2}   \left\|  \psi_{j,c}(\cdot) \int_{-1}^1 \psi_{j,c} (x) u(x)dx   \right\|_{\calL^2([-1,1]) }^2   
   		   \\
   		   & \leq  \dfrac{1}{|\mu_{n,c}|^2}
   		   \sum_{j =0}^n \left\|  \psi_{j,c}(\cdot) \int_{-1}^1 \psi_{j,c} (x) u(x)dx   \right\|_{\calL^2([-1,1]) }^2 
   		   \\
   		    &\leq  \dfrac{1}{|\mu_{n,c}|^2}
   		   \sum_{j =0}^\infty \left\|  \psi_{j,c}(\cdot) \int_{-1}^1 \psi_{j,c} (x) u(x)dx   \right\|_{\calL^2([-1,1]) }^2 
   		     = 
   		     \left(\dfrac{  \|u\|_{ \calL^2([-1,1])} }{|\mu_{n,c}|}\right)^2.  
   	\end{align*}
   	Recalling that $ \|u\|_{ \calL^2([-1,1])} \leq \delta$ (by assumptions) and combining the formulas above, we complete the proof. 
%
%

\subsection{Radon Transform}
 
The Radon transform $\R$ defined   in  \eqref{def:R}
arises in various domains of pure and applied mathematics.
Since  Radon's work \cite{Radon},   this transform  and its applications received significant attention and its properties are  well studied; see, for example, \cite{Naterrer2001} and references therein.
In particular, the Radon transform  $\R[v]$  is closely related to the Fourier transform $\hat v $  (see  \eqref{eq:Fourier}) via the following formula:
\begin{equation}\label{eq:projection}
 	\hat{v} (s \theta ) =  \dfrac{1}{(2\pi)^{d}} \int_{-\infty}^{\infty} e^{i   st} \mathcal{R}  [v] (t,\theta) dt, \qquad  s \in \Reals, \ \theta \in  \mathbb{S}^{d-1}.
\end{equation}
%
In the theory of Radon transform, formula \eqref{eq:projection} is  known as the projection theorem. 
Note that  one can define  the inverse transform $\R^{-1}$ by combining   \eqref{eq:projection}  with   inversion formulas for the 
 Fourier transform: 
 \begin{equation}\label{inv:radon}
\begin{aligned}
  \mathcal{R}^{-1} [u](q)&:= 
  \dfrac{1}{(2\pi)^{d-1}} \int_{\mathbb{S}^{d-1}} \int_{0}^{+\infty}
  e^{-is \theta q} \hat{u}(s, \theta) s^{d-1} ds\, d\theta, \qquad q\in \Reals^d,
 \\ \hat{u}(s, \theta) &:=  \dfrac{1}{2\pi} \int_\Reals e^{ist} u(t, \theta) dt,
 \qquad \qquad s \in \Reals, \ \theta \in \mathbb{S}^{d-1}.
 \end{aligned}
 \end{equation}
  For other inversion formulas for $\R$; see \cite{Radon} and, for example, \cite[Section II.2]{Naterrer2001}.
  
   For real $\nu$, let
 \begin{align*}
 	\calH^{\nu}(\Reals^d)&:=  \{v \st \|v\|_{\calH^{\nu}(\Reals^d)} <\infty\},
     \\
 	\|v\|_{\calH^{\nu}(\Reals^d)}&:= \left(\int_{\Reals^{d}} (1+p^2)^{\nu} |\hat{v}(p)|^2 dp\right)^{1/2},
\\
 	\calH^{\nu}(\Reals \times \mathbb{S}^{d-1} ) &:= \{u \st \|u \|_{\calH^\nu(\Reals \times \mathbb{S}^{d-1} )}<\infty\},\\
 	\|u \|_{\calH^\nu(\Reals \times \mathbb{S}^{d-1} )} &:=  \left(\int_{\mathbb{S}^{d-1}} \int_{-\infty}^{+\infty} (1+s^2)^{\nu}\, |\hat{u}(s,\theta)|^2 ds\, d\theta\right)^{1/2},
 \end{align*}
 where  $v$, $u$ are distributions on $\Reals^{d}$ and  $\Reals \times \mathbb{S}^{d-1}$, respectively.
According to \cite[Theorem 5.1]{Naterrer2001},    if  $v \in \calH^{\nu}(\Reals^{d})$ and $\supp v \subset B_1$  then 
 \begin{equation}\label{eq:H-H}
     a(\nu,d)\|v\|_{\calH^{\nu}(\Reals^{d})} \leq \|\mathcal{R}[v]\|_{\calH^{\nu+ (d-1)/2}(\Reals \times \mathbb{S}^{d-1} ) } \leq
      b (\nu,d)  \|v\|_{\calH^{\nu}(\Reals^{d})}.
 \end{equation}
 In addition,  one can recover  explicit expressions for $a(\nu,d)$
  and  $b(\nu,d)$ from the proof of \cite[Theorem 5.1]{Naterrer2001}. 
We will   also use the following result generalizing  the left inequality  in \eqref{eq:H-H}. 

%
 \begin{Lemma}\label{L:H-H}
 Let   $u \in \calH^{\nu+(d-1)/2}(\Reals \times \mathbb{S}^{d-1})$,
 $\supp u \subseteq [-1,1]\times  \mathbb{S}^{d-1}$, and $u(s,\theta) = u(-s,-\theta)$ for 
 all $(s,\theta)\in \Reals \times \mathbb{S}^{d-1}$. Then, 
 \[
 		 a(\nu,d)
 		 \|v\|_{\calH^{\nu}(\Reals^{d})} 
 		 \leq 
 		 \|u\|_{\calH^{\nu +(d-1)/2}(\Reals \times \mathbb{S}^{d-1} ) },
 \]
 where    $v:=\mathcal{R}^{-1}[u]$ is defined  by \eqref{inv:radon} 
  and $a(\nu,d)$ is the same as in \eqref{eq:H-H}.
 \end{Lemma}
In fact, the proof of Lemma \ref{L:H-H} is identical to the arguments of \cite[Theorem 5.1]{Naterrer2001} for the left inequality  in \eqref{eq:H-H}.    In addition, we use also that 
$u = \mathcal{R}[v]$. 
Note that $v$ defined by \eqref{inv:radon}  might not be compactly supported; see, for example, \cite{Novikov2002} for the asymptotic analysis of $\mathcal{R}^{-1}[u]$ at infinity. 
  
%
%
%
%
%
%
%
%



\section{Stability estimates in 1D}\label{S:1D}
The main result of this section is the following theorem.
\begin{Theorem}\label{T:detailed}
  Let $f,w \in \calL^2([-1,1])$  and 
   $
    	\| \calF_c [f] - w\|_{\calL^2} \leq \delta  
   $ for some $\delta \in  (0,1)$. 
Suppose  that $f \in \calH^\nu([-1,1])$, $\nu> 0$. 
Then, 
\begin{equation}\label{eq:detailed}
\begin{aligned}
		\|f - \calF_{n^*,c}^{-1} [w]\|_{\calL^2([-1,1])}
 	\leq    \gamma_1  c^{-\gamma_2} (1+c)^{\gamma_3} (1+\rho)^{\gamma_4}
    \exp\left(\dfrac{\pi^2c  \log(1+\rho) }{2 e \rho} \right)
      \delta^{1-\alpha} 
      \\ 
     + C(1+c^2)^{\nu/2}  \|f\|_{\calH^\nu([-1,1])}  \left( 2+   \dfrac{ec}{4}\cdot \dfrac{\rho }{\log(1+\rho)}  \right)^{-\nu},
      \end{aligned}
      \end{equation}
      where
$\alpha \in (0,1)$, $\rho = \dfrac{4}{ec}  \alpha  \log  (\delta^{-1})$,    $ n^* = n^*(c,\alpha,\delta) $ is defined by   \eqref{def:n-star},
 $C$ is the constant from \eqref{eq:pi2}, and    $\gamma_1, \gamma_2, \gamma_3,\gamma_4$ are some positive constants independent of $c$, $\alpha$, $\delta$.
\end{Theorem}

 Theorem  \ref{T:detailed} 
follows directly by combining  estimate \eqref{eq:pi1} with $\mu =0$,  estimate \eqref{eq:pi2},
   Lemma~\ref{L:general}, and  the following lemma.
%
\begin{Lemma}\label{L:delta-mu} 
Let  $c$, $\alpha$, $\delta$, $\rho$, $n^*$ be the same as in Theorem \ref{T:detailed}. Then
 \begin{equation}\label{delta-mu}
\dfrac{\delta}{|\mu_{n^*,c}|} \leq   \gamma_1  c^{-\gamma_2} (1+c)^{\gamma_3} (1+\rho)^{\gamma_4}
    \exp\left(\frac{\pi^2c  \log(1+\rho) }{2 e \rho} \right)
      \delta^{1-\alpha}.
\end{equation}
\end{Lemma}
  We prove      Lemma \ref{L:delta-mu} in Section \ref{S:detailed}. 
   The proof of   Lemma \ref{L:delta-mu} is based on   two additional technical lemmas, namely, Lemma~\ref{L:tau} and  Lemma~\ref{L:ass1}.


    Theorem \ref{T:detailed} implies  the following corollary, which  is equivalent to Theorem~\ref{T:multi} in dimension $d=1$.  This corollary  is also crucial for our considerations for  $d\geq 2$ given in Section \ref{S:multi}.
 
 \begin{Corollary}\label{C:main}
 	   Let $f,w \in \calL^2([-1,1])$  and 
   $
    	\| \calF_c [f] - w\|_{\calL^2} \leq \delta   M
   $ for some $\delta \in  (0,1)$ and $M>0$. 
Suppose  that $f \in \calH^\nu([-1,1])$, $\nu> 0$.    
Let  $\alpha \in (0,1)$ and $n^*$ be defined  by \eqref{def:n-star}.
Then, for any    $\beta \in (0, 1-\alpha)$ and any $\mu \in (0,\nu)$,   
\[
	\|f - \calF_{n^*,c}^{-1} [w]\|_{\calL^2([-1,1])} \leq    C_1  M \delta^{\beta} +  C_2    \|f\|_{\calH^\nu([-1,1])}  \left( \log \delta^{-1}\right)^{-\mu},
\]
where  $C_1 = C_1(c,\alpha,\beta)>0$ and $C_2 = C_2(c,\alpha,\nu,\mu)>0$.
 \end{Corollary}
  \begin{proof}
       It is sufficient to prove 
       Corollary \ref{C:main}  for the case $M=1$.   The case $M \neq 1$ is reduced to $M=1$ by scaling $f \rightarrow \tilde{f} = f/M$ and 
       $w \rightarrow \tilde{w} = w/M$.
    Therefore, it remains to show that, under the assumptions of   Theorem \ref{T:detailed}, 
    the following estimate holds   for any    $\beta \in (0, 1-\alpha)$ and any $\mu \in (0,\nu)$:
\begin{equation}\label{eq:inter}
	\|f - \calF_{n^*,c}^{-1} [w]\|_{\calL^2([-1,1])} \leq    C_1  \delta^{\beta} +  C_2  \|f\|_{\calH^\nu([-1,1])}   \left( \log \delta^{-1}\right)^{-\mu}  ,
\end{equation}
where  $C_1 = C_1(c,\alpha,\beta)>0$ and $C_2 = C_2(c,\alpha,\nu,\mu)>0$.

Under our assumptions, we have that:
\begin{align*}
\rho  =  \dfrac{4}{ec}  \alpha  \log  (\delta^{-1})  >0;
 \qquad
  		\dfrac{\pi^2c  \log(1+\rho) }{2 e \rho}  \leq \dfrac{\pi^2c}{ 2e};
    \end{align*} 
    and, for some positive constants  $m_1 = m_1 (c, \alpha,\beta, \gamma_4)$
    and  $m_2  = m_2( c, \alpha, \nu, \mu) $,
  \begin{align*} 
       (1+\rho)^{\gamma_4} \delta^{1-\alpha }  &\leq m_1  \delta^{\beta};
    \\
    \left( 2+   \dfrac{ec}{4}\cdot \dfrac{\rho }{\log(1+\rho)}  \right)^{-\nu} &\leq 
     m_2  \left( \log \delta^{-1}\right)^{-\mu}.
  \end{align*}
   Applying these  estimates in  \eqref{eq:detailed}, we derive   \eqref{eq:inter} 
    with 
    \[
     C_1 =  \gamma_1  c^{-\gamma_2} (1+c)^{\gamma_3} e^{\pi^2c /(2e)}m_1, 
     \qquad C_2 = C (1+c^2)^{\nu/2}m_2.
     \] 
    This completes the proof of Corollary \ref{C:main}. 
    \end{proof}



\section{Multidimensional reconstruction}\label{S:multi}

In this section, we  
 prove Theorems \ref{T1} and    \ref{T:multi}.



\subsection{Proof of Theorem \ref{T1}}\label{S:T1}
 Let $\R[v]$ be  the Radon transform of $v$; see formula \eqref{def:R}.
Since   $\supp v \subset  B_\sigma$, we have that
\begin{equation}\label{Rv=0}
\mathcal{R}  [v] (t,\theta) = 0  \text{ for $|t|>\sigma$.}
 \end{equation}
 Therefore, we only need to integrate over $t \in [-\sigma,\sigma]$    in \eqref{eq:projection}. Then, using the change of variables $s= r x$, $t= \sigma y$ and recalling  $c = r\sigma$, we get that  
 \begin{equation}\label{eq:grtheta}
 	g_{r,\theta} (x) =  \left(\dfrac{2\pi}{\sigma}\right)^d \hat{v} (r x \theta )   = 
 	\dfrac{1}{\sigma^{d-1}}  
 	\int_{-1}^1  	  e^{i  c x y }   \mathcal{R}  [v] (\sigma y,\theta)  dy, \qquad x \in [-1,1]. 
 \end{equation}
 Using  \eqref{Rv=0}, \eqref{eq:grtheta} and
recalling the definitions of   $\calF_c$   and  $f_{r, \sigma}$, 
we obtain that 
 \begin{equation}\label{eq:eq:last}
 		\mathcal{R}  [v] (\sigma y,\theta) = \sigma^{d-1} \calF_c^{-1} [g_{r,\theta}] (y) =  
 		    \sigma^{d-1} f_{r, \sigma}(y, \theta). 
 \end{equation}
 Let
 \begin{equation}\label{def:vsigma} 
 	   v_{\sigma}(q) := v(\sigma q), \qquad q \in \Reals^d.
 \end{equation}
Using   \eqref{def:R} and the change of variables $q = \sigma  q'$, we find that
 \[ 
 \mathcal{R}  [v] (\sigma y,\theta) = \int_{q\in \Reals^d \st q \theta =\sigma y } v(q) dq  = 
     \sigma^{d-1}\int_{q'\in \Reals^d \st q' \theta =y } v(\sigma  q') dq' 
 =
    \sigma^{d-1} \mathcal{R}  [v_{\sigma}](y, \theta).
 \]
 Thus,  also using \eqref{eq:eq:last}, we get
 \begin{equation}\label{eq:vf_sigma}
 \mathcal{R}  [v_{\sigma}]  = f_{r, \sigma}.
 \end{equation}
   Applying the inverse Radon transform   and formula \eqref{def:vsigma} completes the proof.
 
 \subsection{Proof of Theorem \ref{T:multi}}\label{S:multi}
 We will   repeatedly use the following bounds for the Sobolev norm with respect to the argument scaling. 
 \begin{Lemma}\label{L:Snorm}
 		Let $\mathfrak{v} \in \calH^{\eta}(\Reals^d)$ for some $\eta \in \Reals$.
 		  Then,  for any $\sigma>0$,
 		\begin{align*}
  	    \dfrac{\sigma^{\eta - d/2}}{ (1+\sigma)^{\eta}} \|\mathfrak{v} \|_{\calH^{\eta}(\Reals^d)}
  	 \leq \|\mathfrak{v}_{\sigma}\|_{\calH^{\eta}(\Reals^d)} \leq  
  	   \dfrac{(1+\sigma)^{\eta}}{\sigma^{d/2}} \|\mathfrak{v}\|_{\calH^{\eta}(\Reals^d)},
  	   \qquad \text{for $\eta \geq 0$}, 
  	   \\
  	   \dfrac{(1+\sigma)^{\eta}}{\sigma^{d/2}}
  	   \|\mathfrak{v} \|_{\calH^{\eta}(\Reals^d)}
  	 \leq \|\mathfrak{v}_{\sigma}\|_{\calH^{\eta}(\Reals^d)} \leq  
  	     \dfrac{\sigma^{\eta - d/2}}{ (1+\sigma)^{\eta}}   \|\mathfrak{v}\|_{\calH^{\eta}(\Reals^d)},
  	    \qquad \text{for $\eta \leq 0$},
  \end{align*}
  where $\mathfrak{v}_\sigma$ is defined  by $\mathfrak{v}_\sigma (q): = \mathfrak{v}(\sigma q)$, $q \in \Reals^d$.
 \end{Lemma}
 \begin{proof}
 	  Recall that
  	 \begin{align*}
  	 	  \|\mathfrak{v}\|_{\calH^{\eta}(\Reals^d)}^2 &= \int_{\Reals^{d}} (1+  p^2)^{\eta} |\hat{\mathfrak{v}}( p)|^2 dp,\\
 	\|\mathfrak{v}_\sigma \|_{\calH^{\eta}(\Reals^d)}^2 &=  \int_{\Reals^{d}} (1+p^2)^{\eta} |\hat{\mathfrak{v}}_\sigma(p)|^2 dp 
 	\\
 	&=   \sigma^{-2d} \int_{\Reals^{d}} (1+p^2)^{\eta}  |\hat{ \mathfrak{v}}(  p/\sigma)|^2 dp 
 	\\
 	 &=  \sigma^{-d} \int_{\Reals^{d}} (1+ (\sigma p')^2)^{\eta} |\hat{\mathfrak{v}}( p')|^2 dp'. 
 	\end{align*}
 	  Bounding 
 	  	 \begin{align*}
 	 	 \min \{1, \sigma^{2\eta}\}
 	 	\leq
 	 	\dfrac{ (1+ (\sigma p')^2)^{\eta}}{  (1+ (p')^2)^{\eta}}
 	 	\leq  \max \{1, \sigma^{2\eta}\},
 	 \end{align*}
 	  we derive  that 
 	  \[
 	  	  \min\{1,\sigma^{2\eta}\}
 	  	\|\mathfrak{v}\|_{\calH^{\eta}(\Reals^d)}^2 
 	  	\leq \sigma^d \|\mathfrak{v}_\sigma \|_{\calH^{\eta}(\Reals^d)}^2 \leq  \max \{1,\sigma^{2\eta}\}
 	  	  \|\mathfrak{v}\|_{\calH^{\eta}(\Reals^d)}^2. 
 	  \]
 To complete the proof, it remains to observe that
 \begin{align*}
 	\max \{1, \sigma^{2\eta}\} &\leq  
 	\begin{cases}
 		(1+\sigma)^{2\eta}, &\text{if } \eta \geq 0,\\
 		\dfrac{\sigma^{2\eta}}{(1+\sigma)^{2\eta}}, &\text{if } \eta \leq 0,
 	\end{cases}
\\
 	\min \{1, \sigma^{2\eta}\} &\geq  
 	\begin{cases}
 		\dfrac{\sigma^{2\eta}}{(1+\sigma)^{2\eta}} , &\text{if } \eta \geq 0,\\
 		(1+\sigma)^{2\eta} , &\text{if } \eta \leq 0.
 	\end{cases}
 \end{align*}
%
%
%
\end{proof}

 Now we are ready to prove Theorem   \ref{T:multi}.  
 Let $v_{\sigma}$ be defined by \eqref{def:vsigma} and
 \[
 		   v_{\sigma}^{\delta}(q) := v^{\delta}(\sigma q), \qquad q \in \Reals^d.
 \]
  Applying Lemma \ref{L:Snorm} with 
 $\mathfrak{v} = v - v^{\delta}$ and $\eta  = - \frac{d-1}{2}$, we find that
%
%
 \begin{equation}\label{eq1}
 	\|v- v^{\delta}\|_{\calH^{-(d-1)/2}(\Reals^d)} \leq (1+\sigma)^{(d-1)/2} \sigma^{d/2}
 	\|v_{\sigma}-v^{\delta}_{\sigma}\|_{\calH^{-(d-1)/2}(\Reals^d)}.
 \end{equation}
 
  Using  the formulas for  $ f_{r, \sigma}$ and  $ u_{r,\sigma}$
  of Theorems \ref{T1} and  \ref{T:multi},
   we find that
%
\[
	 v_{\sigma} - v_\sigma^\delta  = \mathcal{R}^{-1} [f_{r, \sigma} - u_{r,\sigma}].
\]
Note also that  both $f_{r, \sigma}$ and $ u_{r,\sigma}$ are supported in $[-1,1]\times \mathbb{S}^{d-1}$.  Applying Lemma \ref{L:H-H} 
for $u = f_{r, \sigma} - u_{r, \sigma}$, we get that   
\begin{equation}\label{eq2}
	\|v_{\sigma}-v^{\delta}_{\sigma}\|_{\calH^{-(d-1)/2}(\Reals^d)} \leq 
	\dfrac{1}{a} \|f_{r, \sigma} - u_{r, \sigma}\|_{\calL^2(\Reals\times \mathbb{S}^{d-1})},
\end{equation}
where $a = a(-\tfrac{d-1}{2},d)$ is the constant from \eqref{eq:H-H}.

Observe that
 \begin{equation}\label{eq25}
 	 \|f_{r, \sigma} - u_{r, \sigma}\|_{\calL^2(\Reals\times \mathbb{S}^{d-1})}^2 
 	 = \int_{\mathbb{S}^{d-1}} \|f_{r, \sigma}(\cdot, \theta) - u_{r,\sigma}(\cdot, \theta)\|_{\calL^2([-1,1])}^2 d \theta.
 \end{equation}
 Applying  Corollary~\ref{C:main}  with functions
  $f = f_{r, \sigma}(\cdot, \theta)$ and $w =w_{r,\theta}$, we obtain that, for  any  $\mu \in (0,\nu+\tfrac{d-1}{2})$ and almost all $\theta \in \mathbb{S}^{d-1}$,
 \begin{equation}\label{eq26}
 	\begin{aligned}
	&\|f_{r, \sigma}(\cdot, \theta) - u_{r,\sigma}(\cdot, \theta)\|_{\calL^2([-1,1])} \leq    C_1  M(\theta) \delta^{\beta} +  C_2     H(\theta) \left( \log \delta^{-1}\right)^{-\mu},
 	\\&M(\theta):=   \dfrac{1}{\delta} \|g_{r, \theta} - w_{r,\theta}\|_{\calL^2([-1,1])}, 
  	\\
 	&H(\theta):=  \|f_{r, \sigma}(\cdot, \theta)\|_{\calH^{\nu+(d-1)/2}([-1,1])},
 \end{aligned}
 \end{equation}
 where $f_{r,\sigma}$, $g_{r, \theta}$ and  $ w_{r,\theta}$
 are defined  in Theorems \ref{T1} and  \ref{T:multi},
 $C_1 $ and $C_2$
 are the constants of Corollary~\ref{C:main} with $\nu+ \frac{d-1}{2}$ in place of $\nu$.
%
 Here, the assumption of Corollary~\ref{C:main} that
 \[
 	\|\calF_c[f_{r, \sigma}(\cdot, \theta)] -w_{r,\theta}\|_{\calL^2([-1,1])} \leq  \delta M(\theta)
 \]
 is fulfilled automatically, since   $f_{r, \sigma}(\cdot, \theta) \equiv \calF_c^{-1}[g_{r, \sigma}]$ on $[-1,1]$ by definition.

  In fact, the functions $M$, $H$   belong to $\calL^2(\mathbb{S}^{d-1})$; see formulas  \eqref{eq4} and \eqref{eq45} below.
 Combining formulas \eqref{eq25}, \eqref{eq26}  and the Cauchy--Schwarz inequality
 \[
 \int_{\mathbb{S}^{d-1}} H(\theta) M(\theta) d \theta \leq 
 \|M\|_{\calL^2(\mathbb{S}^{d-1})} 
 \|H\|_{\calL^2(\mathbb{S}^{d-1})}, 
 \]
  we get that
  \begin{equation}\label{eq3}
 \begin{aligned}
 	 \|f_{r, \sigma} - u_{r, \sigma}\|_{\calL^2(\Reals\times \mathbb{S}^{d-1})}^2 
 	 &\leq \int_{\mathbb{S}^{d-1}} \left(C_1  M(\theta) \delta^{\beta} +  C_2    H(\theta)  \left( \log \delta^{-1}\right)^{-\mu}\right)^2 d \theta
 	 \\
 	 &\leq  \left( C_1   \|M\|_{\calL^2(\mathbb{S}^{d-1})}   \delta^{\beta}
 	 +   C_2  
 	  \|H\|_{\calL^2(\mathbb{S}^{d-1})} \left( \log \delta^{-1}\right)^{-\mu}  \right)^2.
 \end{aligned}
 \end{equation}
 Next, we estimate  
 $ \|M\|_{\calL^2(\mathbb{S}^{d-1})}$ and $\|H\|_{\calL^2(\mathbb{S}^{d-1})}$. 
Since
 $\|w-\hat{v}\|_r \leq \delta N$, we get  
 \begin{equation}\label{eq4}
 \begin{aligned}
 	\|M\|_{\calL^2(\mathbb{S}^{d-1})}^2 &=
      \int_{\mathbb{S}^{d-1}} \dfrac{1}{\delta^2} \|g_{r,\theta} - w_{r,\theta}\|_{\calL^2 ([-1,1])}^2  d\theta 
     \\&=
  \dfrac{1}{r\delta^2 }    \left(\dfrac{2\pi}{\sigma}\right)^{2d} 
 	  \int_{\mathbb{S}^{d-1}}  \int_{-r}^r |w(s\theta) - \hat{v}(s\theta) |^2 ds \, d \theta.
    \\
 &=    \dfrac{2}{r\delta^2 }  \left(\dfrac{2\pi}{\sigma}\right)^{2d}    \|w - \hat{v}\|_r^2 \leq  \dfrac{2}{r }  \left(\dfrac{2\pi}{\sigma}\right)^{2d}  N^2.
 \end{aligned}
 \end{equation}
%
In addition,  using \eqref{eq:vf_sigma}, we get
\begin{equation}\label{eq45}
\|H\|_{\calL^2(\mathbb{S}^{d-1})} =   \|f_{r,\sigma}\|_{\calH^{\nu+ (d-1)/2}(\Reals \times \mathbb{S}^{d-1} ) }
= \|\mathcal{R}[v_\sigma]\|_{\calH^{\nu+ (d-1)/2}(\Reals \times \mathbb{S}^{d-1} ) }.
\end{equation}
Using formula \eqref{eq45},  the right inequality of \eqref{eq:H-H},
 and applying Lemma \ref{L:Snorm}
 with   $\mathfrak{v} = v$ and $\eta  = \nu$,  we obtain that 
 \begin{equation}\label{eq5}
 	\|H\|_{\calL^2(\mathbb{S}^{d-1})}    \leq b \|v_\sigma\|_{\calH^{\nu}(\Reals^d)}
 	\leq b  \dfrac{(1+\sigma)^{\nu}}{\sigma^{d/2}} \|v\|_{\calH^{\nu}(\Reals^d)},
 \end{equation}
 where $b= b(\nu,d)$ is the constant from  \eqref{eq:H-H}.
 
 Combining  \eqref{eq1} -- \eqref{eq5}, we derive the required bound \eqref{eq:multi} with 
 \begin{align*}
 	\kappa_1  :=\frac{\sqrt{2} (2\pi)^d  (1+\sigma)^{(d-1)/2} C_1 }{a \sigma^{\frac{d}{2}} \sqrt{r}},
 	\qquad
 	\kappa_2 := \frac{ b}{a }   (1+\sigma)^{\nu + (d-1)/2} C_2. 
 \end{align*}

\section{Proof of  Lemma \ref{L:delta-mu}}\label{S:detailed}

To prove  Lemma \ref{L:delta-mu}, we need two additional technical results given below.
\begin{Lemma}\label{L:tau}
	For any  $\rho > 0$, the equation  
	 \begin{equation}\label{tau:eq}
	 	\tau \log \tau =   \rho
	 \end{equation}
	 has the unique solution $\tau = \tau(\rho)>1$.  Furthermore,  
	 \begin{equation}\label{rho:ineq}
	 	 1 \leq   \dfrac{\rho}{\log (1+\rho)} \leq \tau(\rho) \leq 1+\rho.
	 \end{equation}
\end{Lemma}
\begin{proof}
         Observe that $u_1(\tau) = \tau \log \tau$ is a  strictly increasing continuous function on $[1,+\infty)$,
	$u_1(1) = 0$, and $u_1(\tau) \rightarrow +\infty$ as $\tau \rightarrow +\infty$.  
	Then, by the intermediate value theorem,  equation \eqref{tau:eq} has the 
	 unique solution   $\tau(\rho) \in (0,+\infty)$ for any $\rho>0$.

Next,  note that  
 $u_2(\tau) = \tau  - \tau \log \tau$ is a strictly decreasing function on  $[1,+\infty)$
 since its derivative $u_2'(\tau) = -\log \tau$ is negative for $\tau>1$. Therefore,
 \[\tau(\rho) - \rho = \tau(\rho) - \tau(\rho) \log \tau(\rho) \leq u_2(1) = 1.
\]
Thus,  we proved that $\tau(\rho)  \leq 1 +\rho$.
	Then,  we get $ \log (\tau(\rho) ) \leq \log(1+\rho) $ which implies the  other bound
	\[
	 	\tau(\rho) =  \dfrac{\rho}{ \log (\tau(\rho) )} \geq \dfrac{\rho}{\log (1+ \rho)}.
	 \]
	The remaining inequality   $ \dfrac{\rho}{\log (1+ \rho)}\geq 1$ is equivalent to $e^{\theta}-1 \geq \theta$ with $\theta = \log (1+\rho)$.
\end{proof}

\begin{Lemma}\label{L:ass1}
	Let  $\alpha, \delta \in (0,1)$
	and $\tau $  be  defined according to \eqref{tau:eq} with $\rho = \dfrac{4}{ec}  \alpha  \log  (\delta^{-1})$.
	 Then,   for any $q\geq 0$,  we have 
		\[
	 		e^{ \eta (\log \eta  -    \kappa)} \leq \left(\dfrac{4 \eta}{ c}\right)^{q}   \delta^{-\alpha},
	 	\]
	 	where $\kappa$  is defined according to \eqref{def:kappa} and 
	 	$\eta = \eta(q,\alpha,\delta,c) :=  q+  \tau \dfrac{ec}{4}$.
\end{Lemma}
\begin{proof}
	 First, observe  that 
	 \begin{align*}
	 		\eta (\log \eta -    \kappa) 
	 		&=  (q+ \tau \dfrac{ec}{4}) (\log \eta -    \kappa) 
	 		\\
	 			&=  q (\log \eta -    \kappa)   + \tau \dfrac{ec}{4} (\log (\tau \dfrac{ec}{4})   -   \log (\dfrac{ec}{4}) + \log \eta - \log (\tau \dfrac{ec}{4}))
	 		\\
	 		 &= q (\log  \eta -  \kappa) +   \tau\dfrac{ec}{4} \log \tau +   \tau \dfrac{ec}{4}  (\log  \eta - \log (\tau \dfrac{ec}{4})).
	 \end{align*}
By the definition of $\tau$, we have that  \[
\tau \dfrac{ec}{4} \log \tau =  \alpha \log (\delta^{-1}).\]
Besides,
\[
	 \tau \dfrac{ec}{4} (\log  \eta - \log (\tau \dfrac{ec}{4})  ) =	 \tau \dfrac{ec}{4} \log \left(1 + \dfrac{q }{\tau \tfrac{ec}{4}}\right)  \leq q.
	 \]
Combining the formulas above and recalling the definition of $\kappa$, we derive that 
	 \[
	 	 \eta (\log \eta -   \kappa) \leq 
	 	 q (\log  \eta -  \log (\dfrac{ec}{ 4}))  + \alpha \log (\delta^{-1}) + q =
	 	    q  \log \left(\dfrac{4 \eta}{ c}\right)  +  \alpha \log (\delta^{-1}).
	 \]
	 The required bound follows by exponentiating the  both sides of the last formula.
\end{proof}

%
%

Now, we  are ready to prove Lemma \ref{L:delta-mu}.
First, we combine formulas \eqref{eq:eigenrel} and \eqref{eigenestimate} to get 
\begin{equation}\label{mu_eq}
		|\mu_{n^*,c}|\geq 
		 \sqrt{\dfrac{2\pi}{ c\,A(n^*,c)}}   
			e^{-\tilde n (\log  \tilde{n}-\kappa)},
	\end{equation}
	where  $\tilde{n}   = n^*+\dfrac12$.
Note that   \eqref{eigenestimate} requires $n^*\geq   \max\left\{3, \dfrac {2c}{\pi}\right\}.$
The inequality $n^*\geq 3$ is immediate by the definition of $n^*$. In addition, using that  $\tau>1$ by Lemma \ref{L:tau}, we can estimate 
\[
	n^*\geq 2+ \tau \dfrac{ec}{4} \geq 2+  \dfrac{ec}{4} >  \dfrac{ec}{4}  > \dfrac {2c}{\pi}.
\]
Thus, we justified \eqref{mu_eq}.

Using the  inequalities $1\leq \tau \leq 1+ \rho$ from Lemma \ref{L:tau}, we estimate
\begin{align*}
	n^* \leq 3 + \tau \dfrac{ec}{4} \leq 3(c+1)\tau \leq 3 (c+1) (1+\rho).
\end{align*}
Using the inequality $\tau \geq \dfrac{\rho}{ \log (1+\rho)}$ from Lemma \ref{L:tau}, we also find that
\[
	 e^{(\pi c)^2/4n^*} \leq  e^{ \pi^2 c / (e \tau)  } 
 	 \leq
 	 \exp\left(\dfrac{\pi^2 c \log(1+\rho)}{ e\rho}\right).
\]
Thus, we get that
\begin{equation}\label{mu_eq1}
 \begin{aligned}
	A(n^*,c)&= \nu_1 (n^*)^{\nu_2} \left(\frac{c}{c+1}\right)^{-\nu_3} e^{(\pi c)^2/4n^*}
	\\&\leq \nu_1  3^{\nu_2} (c+1)^{\nu_2 - \nu_3} c^{-\nu_3}(1+\rho)^{\nu_2}\exp\left(\dfrac{\pi^2 c \log(1+\rho)}{ e\rho}\right).
	\end{aligned}
\end{equation}

Similarly as before, using  the  inequalities $1\leq \tau \leq 1+ \rho$ from Lemma \ref{L:tau}, we estimate
\[
	\tilde{n} \leq 3.5 + \tau \dfrac{ec}{4} \leq 3.5 (c+1)(1+\rho).
\]
Then, using Lemma \ref{L:ass1} with $q :=\tilde{n} - \tau \dfrac{ec}{4}$
and observing that $0\leq q\leq 3.5$, we find that
\begin{equation}\label{mu_eq2}
  	e^{\tilde n (\log  \tilde{n}-\kappa)} \leq  \left(\dfrac{4\tilde{n}}{c}\right)^{q}\delta^{-\alpha}
  	\leq  \left(\dfrac{14(c+1)}{c}\right)^{3.5} (1+\rho)^{3.5}\, \delta^{-\alpha}.
  \end{equation}
Substituting the bounds of \eqref{mu_eq1} and of \eqref{mu_eq2} into \eqref{mu_eq}, we derive  estimate \eqref{delta-mu} with 
\[
 \gamma_1 = \sqrt{\dfrac{\nu_1 3^{\nu_2} } {2 \pi}} 14^{3.5},\qquad  \gamma_2 =  \dfrac{\nu_3}{2} +3,
 \qquad \gamma_3 = \dfrac{\nu_2-\nu_3}{2} + 3.5, \qquad \gamma_4 = \dfrac{\nu_2}{2} + 3.5.
\]
Note that if $\gamma_3 \leq  0$ then we can replace it  with zero,
since   $(1+c)^{\gamma_3} \leq 1$  in this case.
This completes the proof of Lemma \ref{L:delta-mu}.

\end{document}